\documentclass[11pt]{article}
\usepackage{mathrsfs}
\usepackage{amssymb}
\usepackage{amsmath}
\usepackage{amsfonts}
\usepackage{graphics}
\usepackage{epsfig}
\usepackage{enumerate}

\usepackage{color}

\allowdisplaybreaks

\newtheorem{theorem}{Theorem}[section]
\newtheorem{lemma}[theorem]{Lemma}
\newtheorem{definition}[theorem]{Definition}

\newtheorem{corollary}[theorem]{Corollary}
\newtheorem{remark}[theorem]{Remark}

\newenvironment{proof}{{\bf Proof:}}{~\hfill $\Box$}
\newenvironment{keywords}{{\bf Keywords: }}{}

\numberwithin{equation}{section}

\begin{document}

\author{Qian Lin $^{1,2}$\footnote{ {\it Email address}:
Qian.Lin@univ-brest.fr}
\\
{\small $^1$School of Mathematics, Shandong University, Jinan
250100,  China;}\\
{\small $^2$ Laboratoire de Math\'ematiques, CNRS UMR 6205,
Universit\'{e} de Bretagne Occidentale,}\\ {\small 6, avenue Victor
Le Gorgeu, CS 93837, 29238 Brest cedex 3, France.}  }

\date{}
 \title{\textbf{Backward doubly stochastic differential
 equations with weak assumptions on the coefficients}}
 \maketitle

\begin{abstract}
In this paper, we deal with one dimensional backward doubly
stochastic differential equations (BDSDEs). We obtain existence
theorems and  comparison theorems for solutions of BDSDEs with weak
assumptions on the coefficients.
\end{abstract}

\vspace{2mm}
 \noindent

\section{Introduction}

 Pardoux and Peng \cite{PP1990} introduced the following
nonlinear backward stochastic differential equations (BSDEs):
\begin{equation*}
 Y_{t}=\xi+\int_{t}^{T}f(s,Y_{s},Z_{s})ds-\int_{t}^{T}Z_{s}dW_{s},\
 t\in [0,T].
   \end{equation*}
 They obtained the existence and uniqueness of solutions under the Lipschitz condition.
 Since then, the theory  of BSDEs  has been developed by many researchers and
there  are many works  attempting to weaken the Lipschitz condition
in order to obtain the existence and uniqueness results of BDSDEs
(see e.g., Bahlali \cite{B}, Briand and  Confortola \cite{BH2008},
Darling and Pardoux \cite{DP},  El Karoui and Huang \cite{EH},
Hamad\`{e}ne \cite{H}, Jia \cite{J2008}, Kobylanski \cite{K2000},
Lepeltier and San Martin \cite{LM1997} and the references therein).
Today the BSDE has become a powerful tool in the study of partial
differential equations, risk measures, mathematical finance, as well
as stochastic optimal controls and stochastic differential games.

   After the nonlinear BSDEs were introduced, Pardoux and Peng \cite{PP1994} brought forward BDSDEs
   with two different directions of stochastic integrals,
i.e., the equations involve both a standard  stochastic It\^o's
integral  and a backward stochastic It\^o's integral:

   \begin{equation}\label{eq}
 Y_{t}=\xi+\int_{t}^{T}f(s,Y_{s},Z_{s})ds+\int_{t}^{T}g(s,Y_{s},Z_{s})dB_{s}-\int_{t}^{T}Z_{s}dW_{s},\
 t\in [0,T],
   \end{equation}
  the integral with respect to $\{B_{t}\}$ is a backward  It\^o's integral
   and the integral with respect to $\{W_{t}\}$ is a standard forward
   It\^o's integral.
By virtue of this kind of BDSDE, Pardoux and Peng \cite{PP1994}
established the connections between certain quasi-linear stochastic
partial differential equations and BDSDEs, and obtained a
probabilistic representation for a class of quasi-linear stochastic
partial differential equations. They established  the existence and
uniqueness results for solutions of BDSDEs under the Lipschitz
condition on the coefficients. This kind of BDSDEs has a practical
background in finance. The extra noise $B$ can be regarded as some
extra information, which can not be detected in the financial
market, but is available to the particular investors.

Since the work of Pardoux and Peng \cite{PP1994}, there are only
several works attempting to relax the Lipschitz condition to get the
existence and uniqueness results for one dimensional BDSDEs.  Shi et
al. \cite{SGL2005} obtained that one dimensional BDSDE (\ref{eq})
has at least one solution if $f$ is continuous and of linear growth
in $(y,z)$, and $\{f(t,0,0)\}_{t\in[0,T]}$ is bounded. Under the
assumptions that $f$ is bounded, left continuous and non-decreasing
in $y$ and Lipschitz in $z$, Lin \cite{Lin} established an existence
theorem for one dimensional BDSDE (\ref{eq}). Lin \cite{LP2008}
proved that one dimensional BDSDE (\ref{eq}) has at least one
solution if the coefficient $f$ is left Lipschitz and left
continuous in $y$, and Lipschitz in $z$. Lin and Wu \cite{LW2008} obtained
a uniqueness result for one dimensional BDSDE (\ref{eq}) under the
conditions that $f$ is Lipschitz in $y$
  and uniformly continuous in $z$.

 Motivated by the above results, one of the objectives of this paper is to get an existence theorem
   for one dimensional BDSDE (\ref{eq}), which generalizes the result in
   Shi et al. \cite{SGL2005} by the condition of the square integrability of $\{f(t,0,0)\}_{t\in[0,T]}$
   instead of  the boundedness of $\{f(t,0,0)\}_{t\in[0,T]}$.
The other objective of this paper is to generalize the existence
result in Lin \cite{LP2008}. We consider the following BDSDE:
 \begin{equation*}\label{}
 Y_{t}=\xi+\int_{t}^{T}\Big(sgn(Y_{s}) Y^{2}_{s}+\sqrt{Z_{s}1_{\{Z_{s}\geq 0\}}}\Big)ds
 +\int_{t}^{T}g(s,Y_{s},Z_{s})dB_{s}-\int_{t}^{T}Z_{s}dW_{s},\
 t\in [0,T].
 \end{equation*}
Since $\sqrt{z1_{\{z\geq 0\}}}$ is not Lipschitz in $z$, then we can
not apply the existence result in Lin \cite{LP2008} to get the
existence theorem of the above BDSDE.
  We shall investigate an existence result for one dimensional BDSDE
   (\ref{eq}) where  $f$ is left Lipschitz and left continuous in $y$  and
uniformly continuous in $z$, which improves the result in  Lin
\cite{LP2008}. Since $f$  is uniformly continuous in $z$, then we
can not apply comparison theorems for solutions of BDSDEs in
\cite{SGL2005} and \cite{LP2008}. In order to get the existence
theorem for solutions of BDSDEs we shall first establish a
comparison theorem for solutions of BDSDEs when $f$ is Lipschitz in
$y$ and uniformly continuous in $z$, which plays an important role.

  This paper is organized as follows: In section 2, we give some preliminaries and notations, which
  will be useful in what follows. In section
  3, we obtain an existence theorem for the solutions of
BDSDEs with continuous coefficients. In section 4, we establish an
existence theorem and a comparison theorem for the solutions of a
class of BDSDEs with discontinuous coefficients.

\section{Preliminaries and Notations}

 Let $T>0$ be a fixed terminal time and $(\Omega,{\cal F},\mathbb{P})$ be a probability
 space. Let $\{ W_t\}_{0 \leq t \leq T}$ and $\{ B_t\}_{0 \leq t \leq
 T}$ be two mutually independent standard Brownian motion processes,
 with values in $\mathbb{R}^{d}$ and  $\mathbb{R}^{l}$, respectively, defined on $(\Omega,{\cal
 F},\mathbb{P})$. Let $\cal{N}$ denote the class of $\cal {P}$-null sets of $\cal
 {F}$. Then, we define
 $${\cal F}_t\doteq {\cal F}_{0,t}^W\vee {\cal F}_{t,T}^B,\quad t\in \left[
0,T\right],$$ where for any process $\{\eta_{t}\}$, ${\cal
F}_{s,t}^{\eta} = \sigma \{\eta_{r}-\eta_{s}, s \leq r \leq t\}\vee
\mathcal {N}$. Let us point out that ${\cal F}_{0,t}^W$ is
increasing and  ${\cal F}_{t,T}^B$ is decreasing in $t$, but ${\cal
F}_t$ is neither increasing nor decreasing in $t$.

Let us introduce the following spaces:

 \vspace{2mm}$ \bullet \ \  L^{2}(\Omega,{\cal F}_{T},\mathbb{P})\doteq \Big\{\xi: {\cal F}_{T}-$measurable
random variable such that  $\mathbb{E}[|\xi|^{2}]<\infty \Big\}$.

 \vspace{2mm}
$\bullet \ \ S^{2}(0,T;\mathbb{R})\doteq \Big\{\varphi:$ $\varphi$
is a continuous process  with value in $\mathbb{R}$  such that
$\|\varphi\|^{2}_{S^{2}}=\mathbb{E}[\sup\limits_{0\leq t \leq T}
|\varphi_{t}|^{2}] <\infty$, and  $\varphi_{t}$ is ${\cal
F}_{t}-$measurable, for all $t\in [0,T]$ \Big\}.

\vspace{2mm}
 $\bullet\ \   M^{2}(0,T;\mathbb{R}^{d})\doteq \Big\{\varphi:$ $\varphi$ is a
 jointly measurable  process  with value in $\mathbb{R}^{d}$  such that
$\|\varphi\|^{2}_{M^{2}}=\mathbb{E}[\int_{0}^{T}|\varphi_{t}|^{2}dt]<\infty,$
and  $\varphi_{t}$ is ${\cal F}_{t}-$measurable, for all $t\in
[0,T]$ \Big\}.

\vspace{2mm}
  Let $$g: \Omega \times [0,T]\times \mathbb{R} \times \mathbb{R}^{d}\rightarrow
  \mathbb{R}^{l}.$$

  In this paper, we suppose that $\xi\in L^{2}(\Omega,{\cal F}_{T},\mathbb{P})$ and $g$ always satisfies the following
  assumptions:

\vspace{2mm}
  $(H1)$ (Lipschitz condition): There exist constants $C>0$ and $0<\alpha<1$ such
  that,  for all $(t,y_{i},z_{i})\in [0,T]\times \mathbb{R} \times \mathbb{R}^{d},$  $i=1, 2$
  $$|g(t,y_{1},z_{1})-g(t,y_{2},z_{2})|^{2}\leq C|y_{1}-y_{2}|^{2}+\alpha|z_{1}-z_{2}|^{2}.$$

  $(H2)$ For all $ (y,z)\in  \mathbb{R} \times \mathbb{R}^{d}$,  $g(\cdot,y, z)\in M^{2}(0,T;\mathbb{R}^l).$

\vspace{2mm}
 Let $$f: \Omega \times [0,T]\times \mathbb{R} \times \mathbb{R}^{d}\rightarrow
  \mathbb{R}$$ be such that, for all $ (t,y,z)\in [0,T]\times\mathbb{R} \times \mathbb{R}^{d}$, $f(t,y,
  z)$ is $ {\cal F}_{t}-$measurable. We make the following
  assumptions:

\vspace{2mm}
  $(H3)$ (Lipschitz condition): There exists a constant $C>0$  such
  that,
  for all $ (t,y_{i},z_{i})\in [0,T]\times \mathbb{R} \times \mathbb{R}^{d},$  $i=1, 2,$
  $$|f(t,y_{1},z_{1})-f(t,y_{2},z_{2})|\leq C(|y_{1}-y_{2}|+|z_{1}-z_{2}|).$$

\vspace{2mm}
  $(H4)$ $f (\cdot,0, 0)\in M^{2}(0,T;\mathbb{R}).$

\vspace{2mm}
  $(H5)$ $f(t,y, \cdot)$ is uniformly continuous and uniformly
  with respect to $(\omega, t, y)$, i.e., there exists a continuous,
  sub-additive, non-decreasing function  $\phi: \mathbb{R}^{+}\rightarrow \mathbb{R}^{+}$
   with linear growth and satisfying $\phi(0)=0$ such that
  $$|f(t,y,z_{1})-f(t,y,z_{2})|\leq \phi(|z_{1}-z_{2}|),$$
  for all $ (t,y,z_{i})\in [0,T]\times \mathbb{R} \times \mathbb{R}^{d},$  $i=1, 2.$
 Here we denote the constant of linear growth for $\phi$ by $C$,
 i.e., $$0\leq \phi (|x|) \leq C(1+|x|),$$  for all $ x \in \mathbb{R}.$

\vspace{2mm}
  $(H6)$ $f(t,\cdot,z)$ is left continuous and satisfies left Lipschitz condition in $y$,
  i.e., for all $ (t,y_{i},z)\in
[0,T]\times \mathbb{R} \times \mathbb{R}^{d}$,  $i=1, 2$ and
$y_{1}\geq y_{2},$
  $$f(t,y_{1},z)-f(t,y_{2},z)\geq -C(y_{1}-y_{2}).$$

\vspace{2mm}
  $(H7)$ For all $(t,\omega)\in [0,T]\times \Omega$, $f(\omega, t,\cdot, \cdot)$ is
  continuous.

\vspace{2mm}
  $(H8)$ There exists a positive constant $C$  such that
  $$|f(\omega,t,y,z)|\leq C(1+|y|+|z|),
  \ \ (\omega, t,y,z)\in\Omega\times[0,T]\times\mathbb{R} \times \mathbb{R}^{d}.$$

\vspace{2mm}
 $(H8')$ There exists a constant $C>0$ and a positive  stochastic process $K\in M^{2}(0,T;\mathbb{R})$ such
   that,  for all $(\omega, t,y,z)\in  \Omega \times [0,T] \times \mathbb{R} \times
   \mathbb{R}^{d}$,
  $$|f(\omega,t,y,z)|\leq C(K_{t}(\omega)+|y|+|z|).$$

\begin{remark}\label{}
Crandall \cite{C} first used $(H5)$ to study viscosity solutions of
partial differential equations.
\end{remark}

\begin{remark}\label{}
From $(H5)$ and $(H6)$ we know that, for $(t,y_{i},z_{i})\in
[0,T]\times \mathbb{R} \times \mathbb{R}^{d}$,  $i=1, 2$ and
$y_{1}\geq y_{2},$ we have
  $$f(t,y_{1},z_{1})-f(t,y_{2},z_{2})\geq -C(y_{1}-y_{2})-\phi(|z_{1}-z_{2}|).$$
\end{remark}

 \begin{remark}\label{}
 If we take $\phi(x)=Cx, x\geq 0,$ in $(H5)$, where $C$ is a positive constant,  then
combining   $(H6)$ with some conditions Lin \cite{LP2008} obtained
that one dimensional BDSDE has at least one solution.
\end{remark}

 \begin{remark}\label{}
Under the assumptions $(H7)$ and $(H8)$  Shi et al. \cite{SGL2005}
proved that one dimensional BDSDE has at least one solution.
\end{remark}

\begin{remark}\label{}
It is obvious that $(H8')$ implies  $(H8)$.
\end{remark}

  For $n\in N$, we let
$$\underline{f}_{n}(t,y,z)=\inf\limits_{u\in \mathbb{R},v\in \mathbb{R}^{d}}\Big\{f(t,u,v)+n(|y-u|+|z-v|)\Big\}$$
and
$$\overline{f}_{n}(t,y,z)=\sup\limits_{u\in \mathbb{R},v\in \mathbb{R}^{d}}\Big\{f(t,u,v)-n(|y-u|+|z-v|)\Big\}.$$
Then, we have the following lemma, which was established by
Lepeltier and San Martin \cite{{LM1997}}.

 \begin{lemma}\label{3}
  If $f$ satisfies $(H7)$ and $(H8)$, then, for
$n>C$ and $(t,y,z)\in  [0,T] \times \mathbb{R} \times
\mathbb{R}^{d}$, we have

(i) $-C(|y|+|z|+1)\leq \underline{f}_{n}(t,y,z)\leq f(t,y,z) \leq
\overline{f}_{n}(t,y,z)\leq C(|y|+|z|+1)$.

(ii) $ \underline{f}_{n}(t,y,z)$ is non-decreasing in $n$ and  $
\overline{f}_{n}(t,y,z)$ is non-increasing in $n$.

(iii) For all $ (\omega, t,y_{i},z_{i})\in \Omega\times[0,T]\times
\mathbb{R} \times \mathbb{R}^{d},$  $i=1, 2,$ we have
  $$|\underline{f}_{n}(\omega,t,y_{1},z_{1})-\underline{f}_{n}(\omega,t,y_{2},z_{2})|\leq n(|y_{1}-y_{2}|+|z_{1}-z_{2}|).$$
The same holds for $\overline{f}_{n}$.

(iv) If $(y_{n}, z_{n})\rightarrow (y,z)$, as $n\rightarrow\infty$,
then $ \underline{f}_{n}(t,y_{n},z_{n})\rightarrow f(t,y,z)$,  as
$n\rightarrow\infty$. The same holds for $\overline{f}_{n}$.
\end{lemma}

  Given $\xi\in L^{2}(\Omega,{\cal F}_{T},\mathbb{P})$, we consider the
  following BDSDE with data $(f,g,T,\xi):$
  \begin{equation}\label{1}
 Y_{t}=\xi+\int_{t}^{T}f(s,Y_{s},Z_{s})ds+\int_{t}^{T}g(s,Y_{s},Z_{s})dB_{s}-\int_{t}^{T}Z_{s}dW_{s}.
   \end{equation}

\begin{definition} A pair of processes $(Y,Z)\in \mathbb{R} \times
\mathbb{R}^{d}$ is called a solution of BDSDE (\ref{1}), if
$(Y,Z)\in S^{2}(0,T;\mathbb{R})\times M^{2}(0,T;\mathbb{R}^{d})$ and
satisfies BDSDE (\ref{1}).
\end{definition}

Pardoux and Peng \cite{PP1994} established the following existence
and uniqueness for solutions of BDSDE (\ref{1}).
\begin{lemma} \label{l4}
Under the assumptions
 $(H1)-(H4)$, BDSDE (\ref{1}) has a unique solution $(Y,Z)\in
S^{2}(0,T;\mathbb{R})\times M^{2}(0,T;\mathbb{R}^{d})$.
\end{lemma}

   Finally, we make another assumption, which will be needed in what
   follows.

\vspace{2mm}
   $(H9)$ There exist two BDSDEs with data $(f_{i},g,T,\xi)$ which
   have at least one solution $(Y^{i},Z^{i})$, $i=1,2$,
   respectively. For all $(t,y,z)\in [0,T]\times \mathbb{R}\times
   \mathbb{R}^{d}$,
   $$f_{1}(t,y,z)\leq f(t,y,z)\leq f_{2}(t,y,z), \quad Y_{t}^{1}\leq Y_{t}^{2},
   a.s..$$
   Moreover, the processes $\Big\{f_{i}(t,Y^{i}_{t},Z^{i}_{t})\Big\}_{t\in[0,T]}$, $i=1,2,$ are
   square integrable.

\section{Existence theorem for BDSDEs
  with general continuous coefficients}
The objective of this section is to obtain an existence theorem for
BDSDEs, which generalizes the corresponding result of Shi et al.
\cite{SGL2005}.

We  first give the following useful lemma. For its proof the reader
is referred to \cite{BHM} and \cite{LM1997}.
 \begin{lemma}\label{l2}
 Let $\underline{f}_{n}$ and $\overline{f}_{n}$ be introduced in Section 2.
 If $f$ satisfies $(H7)$ and $(H8')$, then, for
$n>C$ and  $(t,y,z)\in  [0,T] \times \mathbb{R} \times
\mathbb{R}^{d}$, we have

(i) $-C(|y|+|z|+K_{t})\leq \underline{f}_{n}(t,y,z)\leq f(t,y,z)
\leq \overline{f}_{n}(t,y,z)\leq C(|y|+|z|+K_{t})$.

(ii) $ \underline{f}_{n}(t,y,z)$ is non-decreasing in $n$ and  $
\overline{f}_{n}(t,y,z)$ is non-increasing in $n$.

(iii) For all $ (\omega, t,y_{i},z_{i})\in \Omega\times[0,T]\times
\mathbb{R} \times \mathbb{R}^{d},$  $i=1, 2,$ we have
  $$|\underline{f}_{n}(\omega,t,y_{1},z_{1})-\underline{f}_{n}(\omega,t,y_{2},z_{2})|\leq n(|y_{1}-y_{2}|+|z_{1}-z_{2}|).$$
The same holds for $\overline{f}_{n}$.

(iv) If $(y_{n}, z_{n})\rightarrow (y,z)$, as $n\rightarrow\infty$,
then $ \underline{f}_{n}(t,y_{n},z_{n})\rightarrow f(t,y,z)$,  as
$n\rightarrow\infty$. The same holds for $\overline{f}_{n}$.

\end{lemma}

We also need the following comparison theorem  obtained  in Lin
\cite{LP2008}.
\begin{lemma}\label{2}
Assume BDSDEs (\ref{1}) with data $(f^{1},g,T,\xi^{1})$ and
$(f^{2},g,T,\xi^{2})$ have solutions $(y^{1},z^{1})$ and
$(y^{2},z^{2})$, respectively. If $f^{1}$ satisfies $(H3)$ and
$(H4)$, $\xi^{1}\leq \xi^{2}$, $a.s.$, $f^{1}(t, y_{t}^{2},
z_{t}^{2})\leq f^{2}(t, y_{t}^{2}, z_{t}^{2})$, $d\mathbb{P}dt-a.s.$
(resp. $f^{2}$ satisfies $(H3)$ and $(H4)$, $f^{1}(t, y_{t}^{1},
z_{t}^{1}) \leq f^{2}(t,y_{t}^{1}, z_{t}^{1})$, $d\mathbb{P}dt-a.s.$
), then we have $y_{t}^{1} \leq y_{t}^{2}$, $ a.s.$,  for all $  \ t
\in [0,T]$.
\end{lemma}

We now give the following existence theorem for BDSDEs, which
extends the corresponding result in  Shi et al. \cite{SGL2005} by
eliminating the condition that $\{K_{t}\}_{t\in [0,T]}$ is a bounded
process. The coefficient $g$ in the backward It\^o's integral will
bring the extra estimate difficulty.

\begin{theorem}\label{t1}
Under the assumptions $(H7)$ and $(H8')$, BDSDE  with data $(f, g,
T, \xi)$ has a   minimal (resp. maximal) solution
$(\underline{y},\underline{z})$ (resp.
$(\overline{y},\overline{z})$) of BDSDE with data $(f, g, T, \xi)$,
in the sense that, for any other solution $(y,z)$ of BDSDE with data
$(f, g, T, \xi)$, we have $\underline{y}\leq y$ (resp.
$\overline{y}\geq y$).
\end{theorem}

\begin{proof}
We only prove that BDSDE (\ref{1}) with data $(f, g, T, \xi)$ has a
minimal solution. The other case can be proved similarly. Let
$$h(\omega,t,y,z)=C(K_{t}(\omega)+|y|+|z|),$$
and $\underline{f}_{n}$ be introduced in Section 2.  Then,
$\underline{f}_{n}\leq h$, and we consider the following BDSDEs:
\begin{eqnarray}\label{eq2}
y_{t}^{n}=\xi+\int_{t}^{T}\underline{f}_{n}(s,y_{s}^{n},z_{s}^{n})ds
+\int_{t}^{T}g(s,y_{s}^{n},z_{s}^{n})dB_{s}-\int_{t}^{T}z_{s}^{n}dW_{s}
,\ \ t\in [0,T],
\end{eqnarray}
and
\begin{eqnarray*}
 U_{t}=\xi+\int_{t}^{T}h(s,U_{s},V_{s})ds
 +\int_{t}^{T}g(s,U_{s},V_{s})dB_{s}-\int_{t}^{T}V_{s}dW_{s}, \  \ t\in
 [0,T].
\end{eqnarray*}
From Lemma \ref{l4} it follows that the above BDSDEs have  unique
solutions $(y^{n},z^{n})\in
S^{2}(0,T;\mathbb{R})\times M^{2}(0,T;\mathbb{R}^{d}),$ and $
(U,V)\in S^{2}(0,T;\mathbb{R})\times M^{2}(0,T;\mathbb{R}^{d})$,
respectively.

By a comparison theorem for BDSDEs (see Lemma \ref{2} or
\cite{SGL2005} ) and Lemma \ref{l2} we have, for  $n>C$,
  $$y_{n}\leq y_{n+1}\leq U, \ d\mathbb{P}dt-a.s.$$
Then, there exists a positive constant $A$ independent of $n$ such
that
\begin{eqnarray*}
\parallel U\parallel_{S^{2}}\leq A, \parallel V\parallel_{M^{2}}\leq
A, \ \text{and} \ \parallel y_{n}\parallel_{S^{2}}\leq A.
\end{eqnarray*}
Therefore,  from the dominated convergence theorem it follows that
$\{y_{n}\}$ converges in $S^{2}(0,T;\mathbb{R})$. We shall denote
its limit by $\underline{y}\in S^{2}(0,T;\mathbb{R})$.

  By $(H1)$ and Young inequality we get
\begin{eqnarray*}
|g(t,y_{t}^{n},z_{t}^{n})|^{2}&\leq&(1+\frac{1-\alpha}{4\alpha})|
g(t,y_{t}^{n},z_{t}^{n})-g(t,0,0)|^{2}
+(1+\frac{4\alpha}{1-\alpha})\mid g(t,0,0)\mid^{2}\\
 &\leq& \frac{1+3\alpha}{4\alpha}C|y_{t}^{n}|^{2}+\frac{1+3\alpha}{4}|z_{t}^{n}|^{2}
 +\frac{1+3\alpha}{1-\alpha}|g(t,0,0)|^{2}.
\end{eqnarray*}
By virtue of Lemma \ref{l2} and using Young inequality we have
\begin{eqnarray*}
y_{t}^{n}\underline{f}_{n}(t,y_{t}^{n},z_{t}^{n})&\leq&C|y_{t}^{n}|(K_{t}+|y_{t}^{n}|+|z_{t}^{n}|)\\
 &\leq& (\frac{3C}{2}+\frac{C^{2}}{2-2\alpha})|y_{t}^{n}|^{2}+\frac{1-\alpha}{2}|z_{t}^{n}|^{2}
 +\frac{C}{2}|K_{t}|^{2}.
\end{eqnarray*}
Consequently, by the above inequalities and applying It\^o's formula
to $|y_{t}^{n}|^{2}$ and taking mathematical expectation, we obtain
 \begin{eqnarray*}
\mathbb{E}\int_{0}^{T}|z_{t}^{n}|^{2}dt&=&\mathbb{E}|\xi|^{2}-|
y_{0}^{n}|^{2}
  +2\mathbb{E}\int_{0}^{T}y_{t}^{n}\underline{f}_{n}(t,y_{t}^{n},z_{t}^{n})dt\\
  &&+\mathbb{E}\int_{0}^{T}| g(t,y_{t}^{n},z_{t}^{n})|^{2}dt\\
   &\leq&
   \mathbb{E}|\xi|^{2}+\frac{3+\alpha}{4}\mathbb{E}\int_{0}^{T}|z_{t}^{n}|^{2}dt+
   \frac{1+3\alpha}{1-\alpha} \mathbb{E}\int_{0}^{T}|g(t,0,0)|^{2}dt\\
   &&+\Big(\frac{3C}{2}+\frac{C^{2}}{2-2\alpha}+\frac{1+3\alpha}{4\alpha}C\Big)
   \mathbb{E}\int_{0}^{T}|y_{t}^{n}|^{2}dt.
   +\frac{C}{2} E\int_{0}^{T}|K_{t}|^{2}dt.
\end{eqnarray*}
Therefore,
 \begin{eqnarray*}
\mathbb{E}\int_{0}^{T}|z_{t}^{n}|^{2}dt &\leq&
   \frac{4}{1-\alpha}E|\xi|^{2}+
    \frac{4+12\alpha}{(1-\alpha)^{2}} \mathbb{E}\int_{0}^{T}|g(t,0,0)|^{2}dt\\
   &&+  \frac{4}{1-\alpha}\Big(\frac{3C}{2}+\frac{C^{2}}{2-2\alpha}+\frac{1+3\alpha}{4\alpha}C\Big)
   \mathbb{E}\int_{0}^{T}|y_{t}^{n}|^{2}dt
   +  \frac{2C}{1-\alpha} \mathbb{E}\int_{0}^{T}|K_{t}|^{2}dt\\
   &\doteq& A,
\end{eqnarray*}
which is bounded and independent of $n$.

Using It\^o's formula to $\mid y_{t}^{n}-y_{t}^{m}\mid^{2}$ we
obtain
\begin{eqnarray*}
 &&\mathbb{E}\int_{0}^{T}|z_{t}^{n}-z_{t}^{m}|^{2}dt+\mid
 y_{0}^{n}-y_{0}^{m}\mid^{2}\\
  &=&2\mathbb{E}[\int_{0}^{T}(y_{t}^{n}-y_{t}^{m})
  \Big(\underline{f}_{n}(t,y_{t}^{n},z_{t}^{n})-\underline{f}_{m}(t,y_{t}^{m},z_{t}^{m})\Big)dt]\\
  &&+\mathbb{E}\int_{0}^{T}\mid
  g(t,y_{t}^{n},z_{t}^{n})-g(t,y_{t}^{m},z_{t}^{m})\mid^{2}dt.
\end{eqnarray*}
From Lemma \ref{l2}, $\parallel z_{n} \parallel_{M^{2}}\leq A$ and
$\parallel y_{n}\parallel_{S^{2}}\leq A$ it follows that there
exists a positive constant $C_{0}$ independent of $n,m$ such that
\begin{eqnarray*}
\begin{aligned}
  &\mathbb{E}[\int_{0}^{T}(y_{t}^{n}-y_{t}^{m})
  (\underline{f}_{n}(t,y_{t}^{n},z_{t}^{n})-\underline{f}_{m}(t,y_{t}^{m},z_{t}^{m}))dt]\\
  &\leq C_{0}(\mathbb{E}\int_{0}^{T}|y_{t}^{n}-y_{t}^{m}|^{2}dt)^{\frac{1}{2}}.
\end{aligned}
\end{eqnarray*}
Therefore, by virtue of $(H1)$  we get
\begin{eqnarray*}
 &&\mathbb{E}\int_{0}^{T}|z_{t}^{n}-z_{t}^{m}|^{2}dt+\mid
 y_{0}^{n}-y_{0}^{m}\mid^{2}\\
   &\leq&
   C_{0}\Big\{\mathbb{E}\int_{0}^{T}|y_{t}^{n}-y_{t}^{m}|^{2}dt\Big\}^{\frac{1}{2}}\\
   &&+\alpha \mathbb{E}\int_{0}^{T}|z_{t}^{n}-z_{t}^{m}|^{2}dt
   +C\mathbb{E}\int_{0}^{T}|y_{t}^{n}-y_{t}^{m}|^{2}dt.
\end{eqnarray*}
Then, we deduce
\begin{eqnarray*}
  &&(1-\alpha)\mathbb{E}\int_{0}^{T}|z_{t}^{n}-z_{t}^{m}|^{2}dt\\
   &&\leq C_{0}\Big\{\mathbb{E}\int_{0}^{T}|y_{t}^{n}-y_{t}^{m}|^{2}dt\Big\}^{\frac{1}{2}}
   +C\mathbb{E}\int_{0}^{T}|y_{t}^{n}-y_{t}^{m}|^{2}dt.
\end{eqnarray*}
Therefore, $\{z^{n}\}_{n=1}^{\infty}$ is a Cauchy sequence in
$M^{2}(0,T;\mathbb{R}^{d})$. Then, there exists $\underline{z}\in
M^{2}(0,T;\mathbb{R}^{d})$ such that
$$\lim\limits_{n\rightarrow\infty}
\mathbb{E}[\int_{0}^{T}|z_{t}^{n}-\underline{z}_{t}|^{2}dt]=0.$$
Thanks to $(H1)$ and BDG inequality we know that there exists a
positive constant $C_{1}$ independent of $n$ such that
\begin{eqnarray*}
&&\mathbb{E}[\sup\limits_{t\in[0,T]}|\int_{t}^{T}g(s,\underline{y}_{s},\underline{z}_{s})dB_{s}
-\int_{t}^{T}g(s,y_{s}^{n},z_{s}^{n})dB_{s}|^{2}]\\
&\leq&C_{1}\mathbb{E}[\int_{0}^{T}|g(s,\underline{y}_{s},\underline{z}_{s})
-g(s,y_{s}^{n},z_{s}^{n})|^{2}ds]\\
&\leq&C_{1}\alpha\mathbb{E}[\int_{0}^{T}|z_{t}^{n}-\underline{z}_{t}|^{2}dt]+
C_{1}C\mathbb{E}[\int_{0}^{T}|y_{t}^{n}-\underline{y}_{t}|^{2}dt]\\
&\rightarrow& 0,\ \text{as}\ n\rightarrow\infty.
\end{eqnarray*}
 For all $N>0$ and $(t,y,z)\in
[0,T]\times \mathbb{R} \times \mathbb{R}^{d}$, from Lemma \ref{l2}
and Dini's Theorem it follows that
\begin{eqnarray*}
  \lim\limits_{n\rightarrow\infty}\sup\limits_{|y|+|z|\leq
  N}|\underline{f}_{n}(t,y,z)-f(t,y,z)|=0, \ dtd\mathbb{P}-a.s.
\end{eqnarray*}
Therefore, by the dominated convergence theorem we have
\begin{eqnarray*}
  \lim\limits_{n\rightarrow\infty}
\mathbb{E}\int_{0}^{T}|\underline{f}_{n}(t,y_{t}^{n},z_{t}^{n})-f(t,y_{t}^{n},z_{t}^{n})|
1_{\{|y_{t}^{n}|+|z_{t}^{n}|\leq N\}}dt=0.
\end{eqnarray*}
By virtue of $(H7)$ we know that $$\mathbb{E}\int_{0}^{T}|f
(t,y_{t}^{n},z_{t}^{n})-f(t,\underline{y}_{t},\underline{z}_{t})|
1_{\{|y_{t}^{n}|+|z_{t}^{n}|\leq N\}}dt$$ converges to $0$ at least
along a subsequence.

From Lemma \ref{l2} and $(H8')$ it follows that
\begin{eqnarray*}
 \mathbb{E}\int_{0}^{T}|\underline{f}_{n}(t,y_{t}^{n},z_{t}^{n})
 -f(t,\underline{y}_{t},\underline{z}_{t})|(|y_{t}^{n}|+|z_{t}^{n}|)dt
\doteq C_{2}<\infty.
\end{eqnarray*}
Here $C_{2}$ is a positive constant and independent of $n$.
Consequently,
\begin{eqnarray*}
 \mathbb{E}\int_{0}^{T}|\underline{f}_{n}(t,y_{t}^{n},z_{t}^{n})-f(t,\underline{y}_{t},\underline{z}_{t})|
1_{\{|y_{t}^{n}|+|z_{t}^{n}|> N\}}dt\leq \frac{C_{2}}{N}.
\end{eqnarray*}
Combining the above inequalities, passing to a subsequence if
necessary,  we have
\begin{eqnarray*}
&&\mathbb{E}\int_{0}^{T}|\underline{f}_{n}(t,y_{t}^{n},z_{t}^{n})-f(t,\underline{y}_{t},\underline{z}_{t})|dt\\
&\leq&
\mathbb{E}\int_{0}^{T}|\underline{f}_{n}(t,y_{t}^{n},z_{t}^{n})-f(t,y_{t}^{n},z_{t}^{n})|
1_{\{|y_{t}^{n}|+|z_{t}^{n}|\leq N\}}dt\\
&&+\mathbb{E}\int_{0}^{T}|f
(t,y_{t}^{n},z_{t}^{n})-f(t,\underline{y}_{t},\underline{z}_{t})|
1_{\{|y_{t}^{n}|+|z_{t}^{n}|\leq N\}}dt+\frac{C_{2}}{N}\\
&\rightarrow&\frac{C_{2}}{N},
\end{eqnarray*}
as $n\rightarrow\infty$. Thus, letting $N\rightarrow\infty,$ we have
\begin{eqnarray*}
\mathbb{E}\int_{0}^{T}|\underline{f}_{n}(t,y_{t}^{n},z_{t}^{n})-f(t,y_{t},z_{t})|dt
\rightarrow0,
\end{eqnarray*}
as $n\rightarrow\infty,$ passing to a subsequence if necessary.  We
now pass to the limit on both sides of BDSDE (\ref{eq2}),  passing
to a subsequence if necessary, it follows that
\begin{eqnarray*}\label{}
\underline{y}_{t}=\xi+\int_{t}^{T}f(s,\underline{y}_{s},\underline{z}_{s})ds
+\int_{t}^{T}g(s,\underline{y}_{s},\underline{z}_{s})dB_{s}-\int_{t}^{T}\underline{z}_{s}dW_{s}.
\end{eqnarray*}
Consequently,  BDSDE  with data $(f, g, T, \xi)$ has a solution
$(\underline{y},\underline{z})$.

Let $(y',z')$ be any solution of BDSDE with data $(f, g, T, \xi)$.
Then, let us consider  the following BDSDEs:
\begin{eqnarray*}\label{}
y'_{t}=\xi+\int_{t}^{T}f(s,y'_{s},z'_{s})ds
+\int_{t}^{T}g(s,y'_{s},z'_{s})dB_{s}-\int_{t}^{T}z'_{s}dW_{s},  \
t\in [0,T],
\end{eqnarray*}
and
\begin{eqnarray*}
y_{t}^{n}=\xi+\int_{t}^{T}\underline{f}_{n}(s,y_{s}^{n},z_{s}^{n})ds
+\int_{t}^{T}g(s,y_{s}^{n},z_{s}^{n})dB_{s}-\int_{t}^{T}z_{s}^{n}dW_{s}
,\ \ t\in [0,T],
\end{eqnarray*}
By virtue of Lemma \ref{2}  we have $y^{n} \leq y'.$ Consequently,
due to the first part of the proof and taking the limit we have
$\underline{y} \leq y'.$
 The proof is complete.
\end{proof}\\

If $\{K_{t}\}_{t\in [0,T]}$ is a bounded process, then we have the
following corollary, which was obtained by Shi et al.
\cite{SGL2005}.

\begin{corollary} \label{}
Under the assumptions
 $(H7)$ and $(H8)$, BDSDE with data $(f,g,T,\xi)$ has the minimal solution
 $(\underline{y}_{t},\underline{z}_{t})_{0 \leq t \leq T}$
 (resp. maximal solution $(\overline{y}_{t},\overline{z}_{t})_{0 \leq t \leq
 T}$). Moreover, for all $t\in[0,T]$,
  $$\underline{y}_{t}^{n}\leq \underline{y}_{t}^{n+1}\leq \underline{y}_{t}\leq \overline{y}_{t}
  \leq\overline{y}_{t}^{n+1}\leq\overline{y}_{t}^{n}.$$
 And $(\underline{y}^{n},\underline{z}^{n})\rightarrow
  (\underline{y},\underline{z})$  and $(\overline{y}^{n},\overline{z}^{n})
  \rightarrow (\overline{y},\overline{z})$ both in $S^{2}(0,T;\mathbb{R})\times
  M^{2}(0,T;\mathbb{R}^{d})$, as $n\rightarrow \infty$, where $(\underline{y}^{n},\underline{z}^{n})$ is the
  unique solution of BDSDE  with data $(\underline{f}_{n},g,T,\xi)$ and
  $(\overline{y}^{n},\overline{z}^{n})$ is the
  unique solution of BDSDE  with data $(\overline{f}_{n},g,T,\xi)$.
\end{corollary}

\section{Existence theorem and comparison theorems for BDSDEs
  with discontinuous coefficients}

  The objective of this section is to investigate  an  existence theorem and a comparison theorem
  for solutions of BDSDEs with discontinuous coefficients.

We shall give a comparison theorem for BDSDEs (\ref{1}) under the
conditions that $f$ is Lipschitz in $y$ and uniformly continuous in
$z$, i.e.,

 \vspace{2mm}
   $(H10)$ There exists a positive constant $C$  such that,
  for all $(t,y_{i},z_{i})\in [0,T]\times \mathbb{R} \times \mathbb{R}^{d},$  $i=1, 2,$
  $$|f(t,y_{1},z_{1})-f(t,y_{2},z_{2})|\leq C|y_{1}-y_{2}|+\phi(|z_{1}-z_{2}|),$$
where $\phi$ is introduced in $(H5)$.

We  need the following existence theorem and uniqueness theorem for
BDSDEs, which was established in \cite{LW2008}.
\begin{lemma} \label{5}
 Under the assumptions $(H4)$ and $(H10)$, BDSDE (\ref{1}) has a unique solution $(Y,Z)\in
S^{2}(0,T;\mathbb{R})\times M^{2}(0,T;\mathbb{R}^{d})$.
\end{lemma}

 Since $\phi$ is uniformly
continuous, then we can not apply  comparison theorems for solutions
of BDSDEs in \cite{SGL2005} and  \cite{LP2008} to the proofs of
Lemma \ref{7} and Theorem \ref{10}. We now establish a comparison
theorem of BDSDEs when $f$ satisfies the condition $(H10)$, which
plays an important role in the proofs of Lemma \ref{7} and Theorem
\ref{10}.

\begin{theorem}\label{th2}
Suppose that  BDSDEs  with data $(f^{1},g,T,\xi^{1})$ and
$(f^{2},g,T,\xi^{2})$ have solutions $(y^{1},z^{1})$ and
$(y^{2},z^{2})$, respectively. If $f^{1}$ satisfies $(H4)$ and
$(H10)$, $\xi^{1}\leq \xi^{2}$, $a.s.$, $f^{1}(t, y_{t}^{2},
z_{t}^{2})\leq f^{2}(t, y_{t}^{2}, z_{t}^{2})$, $d\mathbb{P}dt-a.s.$
(resp. $f^{2}$ satisfies $(H4)$ and $(H10)$, $f^{1}(t, y_{t}^{1},
z_{t}^{1}) \leq f^{2}(t,y_{t}^{1}, z_{t}^{1})$,
$d\mathbb{P}dt-a.s.$), then we have $y_{t}^{1} \leq y_{t}^{2}$, $
a.s.$,  for all $  \ t \in [0,T]$.
\end{theorem}

\begin{proof}
We only prove the first case, the other case can be proved
similarly. For $n\in \mathbb{N}, (t,y,z)\in[0,T]\times
\mathbb{R}\times \mathbb{R}^{d}$, we let
$$f^{1}_{n}(t,y,z)=\inf\limits_{v\in \mathbb{R}^{d}}\Big\{f^{1}(t,y,v)+n|z-v|\Big\}.$$
Then, since $f^{1}(t, y_{t}^{2}, z_{t}^{2})\leq f^{2}(t, y_{t}^{2},
z_{t}^{2})$, $a.s.,$ we have $f^{1}_{n}(t, y_{t}^{2}, z_{t}^{2})\leq
f^{1}(t, y_{t}^{2}. z_{t}^{2})\leq f^{2}(t, y_{t}^{2}, z_{t}^{2})$,
$a.s.$

\vspace{2mm} From Lemma \ref{3} and $(H10)$ it follows that, for all
$(t,y_{i},z_{i})\in [0,T]\times \mathbb{R} \times \mathbb{R}^{d},$
$i=1, 2,$
  $$|f^{1}_{n}(t,y_{1},z_{1})-f^{1}_{n}(t,y_{2},z_{2})|\leq C|y_{1}-y_{2}|+n|z_{1}-z_{2}|,$$
and we consider the following BDSDE:
\begin{eqnarray*}
\underline{y}_{t}^{n}=\xi^{1}+\int_{t}^{T}f^{1}_{n}(s,\underline{y}_{s}^{n},\underline{z}_{s}^{n})ds
+\int_{t}^{T}g(s,\underline{y}_{s}^{n},\underline{z}_{s}^{n})dB_{s}-\int_{t}^{T}\underline{z}_{s}^{n}dW_{s}
, \ t\in [0,T],
 \end{eqnarray*}
and
\begin{eqnarray*}
 y^{2}_{t}=\xi^{2}+\int_{t}^{T}f^{2}(s,y^{2}_{s},z^{2}_{s})ds
 +\int_{t}^{T}g(s,y^{2}_{s},z^{2}_{s})dB_{s}-\int_{t}^{T}z^{2}_{s}dW_{s}, \  t\in
 [0,T].
\end{eqnarray*}
  By virtue of  Lemma \ref{2}  we obtain $\underline{y}_{t}^{n}\leq
y_{t}^{2}$, for $n>C$.  Theorem \ref{t1} and Lemma \ref{5}  yield
$$y_{t}^{1} \leq y_{t}^{2}, \quad a.s., \ \text{for all}\ \ t \in [0,T].$$
 The proof is complete.
\end{proof}\\

From now we study an  existence theorem and a comparison theorem
  for solutions of BDSDEs under the conditions  $ (H5),(H6)$ and $(H9)$.

\vspace{2mm} From  $(H9)$  we know that there exist two BDSDEs:
$i=1,2,$
\begin{eqnarray*}
Y_{t}^{i}=\xi+\int_{t}^{T}f_{i}(s,Y_{s}^{i},Z_{s}^{i})ds
+\int_{t}^{T}g(s,Y_{s}^{i},Z_{s}^{i})dB_{s}-\int_{t}^{T}Z_{s}^{i}dW_{s}
, \ t\in [0,T]
\end{eqnarray*}
such that $f_{i}(\cdot,Y_{\cdot}^{i},Z_{\cdot}^{i})\in
M^{2}(0,T;\mathbb{R})$.

  \vspace{2mm}We now construct a sequence of BDSDEs as follows:
\begin{eqnarray}\label{6}
\begin{aligned}
 y_{t}^{n}=\xi&+\int_{t}^{T}[f(s,y_{s}^{n-1},z_{s}^{n-1})-C(y_{s}^{n}-y_{s}^{n-1})
 -\phi(|z_{s}^{n}-z_{s}^{n-1}|)]ds \\
&+\int_{t}^{T}g(s,y_{s}^{n},z_{s}^{n})dB_{s}-\int_{t}^{T}z_{s}^{n}dW_{s}
,\  \ t\in [0,T],
\end{aligned}
\end{eqnarray}
where $n=1,2,\cdots,$ and $(y^{0},z^{0})=(Y^{1},Z^{1})$. We have the
following lemma:

\begin{lemma}\label{7}
 Under the assumptions $(H5),(H6)$ and $(H9)$, for all
$n=1,2,\cdots,$  BDSDE (\ref{6}) has a unique solution
$(y^{n},z^{n})\in S^{2}(0,T;\mathbb{R})\times
M^{2}(0,T;\mathbb{R}^{d})$,  and $Y_{t}^{1}\leq y_{t}^{n}\leq
y_{t}^{n+1}\leq Y_{t}^{2}$, a.s., for all $t\in [0,T]$.
\end{lemma}

 \begin{proof} For $n=1$, from $(H5)$, $(H6)$, $(H9)$ and $Y_{t}^{1}\leq
 Y_{t}^{2}$  it follows that
\begin{eqnarray*}
&&f_{2}(t,Y_{t}^{2},Z_{t}^{2})-f(t,Y_{t}^{1},Z_{t}^{1})\geq
 f(t,Y_{t}^{2},Z_{t}^{2})-f(t,Y_{t}^{1},Z_{t}^{1})\\
 && \geq-C(Y_{t}^{2}-Y_{t}^{1})- \phi(|
 Z_{t}^{2}-Z_{t}^{1}|).
\end{eqnarray*}
 Then,  we have
 $$f_{2}(t,Y_{t}^{2},Z_{t}^{2})+C(Y_{t}^{2}-Y_{t}^{1})+\phi(|
 Z_{t}^{2}-Z_{t}^{1}|) \geq f(t,Y_{t}^{1},Z_{t}^{1})\geq f_{1}(t,Y_{t}^{1},Z_{t}^{1}), $$
and
 $$f_{2}(t,Y_{t}^{2},Z_{t}^{2})\geq f(t,Y_{t}^{1},Z_{t}^{1})
  -C(Y_{t}^{2}-Y_{t}^{1})- \phi(| Z_{t}^{2}-Z_{t}^{1}|).$$
Thus, due to $(H9)$ and linear growth of $\phi$ we have
$f(\cdot,Y_{\cdot}^{1},Z_{\cdot}^{1})\in M^{2}(0,T;\mathbb{R})$, and
from Lemma \ref{6} it follows that BDSDE (\ref{6}) has a unique
solution $(y^{1},z^{1})\in S^{2}(0,T;\mathbb{R})\times
M^{2}(0,T;\mathbb{R}^{d})$, and by virtue of Theorem \ref{th2} we
have $Y^{1}_{t}\leq y^{1}_{t}\leq Y^{2}_{t}$,  a.s., for all
$t\in[0,T]$.

\vspace{2mm} For $n=2$, by $(H5)$, $(H6)$, $(H9)$ and
   $Y^{1}_{\cdot}\leq y^{1}_{\cdot}\leq Y^{2}_{\cdot}$   we deduce
\begin{eqnarray*}
  f_{2}(t,Y_{t}^{2},Z_{t}^{2})-f(t,y_{t}^{1},z_{t}^{1})&\geq&
 f(t,Y_{t}^{2},Z_{t}^{2})-f(t,y_{t}^{1},z_{t}^{1})\\
 &\geq& -C(Y_{t}^{2}-y_{t}^{1})- \phi(|
 Z_{t}^{2}-z_{t}^{1}|),
 \end{eqnarray*}
and
\begin{eqnarray*}
f(t,y_{t}^{1},z_{t}^{1})-f_{1}(t,Y_{t}^{1},Z_{t}^{1})&\geq&
 f(t,y_{t}^{1},z_{t}^{1})-f(t,Y_{t}^{1},Z_{t}^{1})\\
&\geq& -C(y_{t}^{1}-Y_{t}^{1})- \phi(|
 z_{t}^{1}-Z_{t}^{1}|).
 \end{eqnarray*}
 Then,   we obtain
 \begin{eqnarray*}
 \begin{aligned}
&f_{2}(t,Y_{t}^{2},Z_{t}^{2})+C(Y_{t}^{2}-y_{t}^{1})+\phi(|
 Z_{t}^{2}-z_{t}^{1}|)\\&\geq f(t,y_{t}^{1},z_{t}^{1})
 \geq f_{1}(t,Y_{t}^{1},Z_{t}^{1})-C(y_{t}^{1}-Y_{t}^{1})-\phi(|
 z_{t}^{1}-Z_{t}^{1}|),
 \end{aligned}
\end{eqnarray*}
 $$f_{2}(t,Y_{t}^{2},Z_{t}^{2})\geq f(t,y_{t}^{1},z_{t}^{1})
  -C(Y_{t}^{2}-y_{t}^{1})- \phi(| Z_{t}^{2}-z_{t}^{1}|),$$
  and
$$f(t,y_{t}^{1},z_{t}^{1})\geq f(t,Y_{t}^{1},Z_{t}^{1})
  -C(y_{t}^{1}-Y_{t}^{1})- \phi(|z_{t}^{1}-Z_{t}^{1}|).$$
Thus, $f(\cdot,y_{\cdot}^{1},z_{\cdot}^{1})\in
M^{2}(0,T;\mathbb{R})$, and by Lemma \ref{5} and  Theorem \ref{th2}
we know that BDSDE (\ref{6}) has a unique solution $(y^{2},z^{2})\in
S^{2}(0,T;\mathbb{R})\times M^{2}(0,T;\mathbb{R}^{d})$, and
$Y^{1}_{t}\leq y^{1}_{t}\leq y^{2}_{t} \leq Y^{2}_{t}$,  a.s., for
all $t\in[0,T]$.

\vspace{2mm}
 For $n>2$, we suppose that $Y^{1}\leq y^{n-1}\leq y^{n} \leq
 Y^{2}$, and $f(\cdot,y_{\cdot}^{n-1},z_{\cdot}^{n-1})\in M^{2}(0,T;\mathbb{R})$.
 Let us consider the following BDSDE:
 \begin{eqnarray}\label{11}
 \begin{aligned}
 y_{t}^{n+1}=\xi&+\int_{t}^{T}\Big[f(s,y_{s}^{n},z_{s}^{n})-C(y_{s}^{n+1}-y_{s}^{n})
 -\phi(|z_{s}^{n+1}-z_{s}^{n}|)\Big]ds  \\
&+\int_{t}^{T}g(s,y_{s}^{n+1},z_{s}^{n+1})dB_{s}-\int_{t}^{T}z_{s}^{n+1}dW_{s}
,\  t\in [0,T].
 \end{aligned}
\end{eqnarray}
Using the similar argument as $n=2$  we get
 \begin{eqnarray}\label{8}
 \begin{aligned}
 &f_{2}(t,Y_{t}^{2},Z_{t}^{2})+C(Y_{t}^{2}-y_{t}^{n})+\phi(|
 Z_{t}^{2}-z_{t}^{n}|) \geq f(t,y_{t}^{n},z_{t}^{n})\\
 & \geq f_{1}(t,Y_{t}^{1},Z_{t}^{1})-C(y_{t}^{n}-Y_{t}^{1})-\phi(|
 z_{t}^{n}-Z_{t}^{1}|),
 \end{aligned}
\end{eqnarray}
 $$f_{2}(t,Y_{t}^{2},Z_{t}^{2})\geq f(t,y_{t}^{n},z_{t}^{n})
  -C(Y_{t}^{2}-y_{t}^{n})- \phi(| Z_{t}^{2}-z_{t}^{n}|),$$
and
$$f(t,y_{t}^{n},z_{t}^{n})\geq f(t,y_{t}^{n-1},z_{t}^{n-1})
  -C(y_{t}^{n}-y_{t}^{n-1})- \phi(| z_{t}^{n}-z_{t}^{n-1}|).$$
Consequently, $f(\cdot,y_{\cdot}^{n},z_{\cdot}^{n})\in
M^{2}(0,T;\mathbb{R})$, and using  Lemma \ref{5} and  Theorem
\ref{th2} again we obtain that BDSDE (\ref{11}) has a unique
solution $(y^{n+1},z^{n+1})\in S^{2}(0,T;\mathbb{R})\times
M^{2}(0,T;\mathbb{R}^{d})$,  and $Y^{1}_{t}\leq y^{n}_{t}\leq
y^{n+1}_{t} \leq Y^{2}_{t}$,  a.s., for all $t\in[0,T]$. The proof
is complete.
\end{proof}

\vspace{2mm} We now state and prove the main result in this section.

\begin{theorem}\label{10}
Under the assumptions   $(H5),(H6)$ and $(H9)$, BDSDE  with data
$(f, g, T, \xi)$ has a solution. Moreover, if $f_{1}$ satisfies
$(H4)$ and $(H10)$, then  BDSDE with data $(f, g, T, \xi)$ has a
minimal solution $(\underline{y},\underline{z})$, in the sense that,
for any other solution $(y,z)$ of BDSDE with data $(f, g, T, \xi)$,
we have $\underline{y}\leq y$.
\end{theorem}

 \begin{proof} By Lemma \ref{7} we know that
 $\{y^{n}\}_{n=1}^{\infty}$ converges to a limit $\underline{y}$ in
 $S^{2}(0,T;\mathbb{R})$ and $$\sup\limits_{n}\mathbb{E}[\sup\limits_{0\leq t \leq T}| y_{t}^{n}|^{2}]
 \leq \mathbb{E}[\sup\limits_{0\leq t \leq T}| Y_{t}^{1}|^{2}]+\mathbb{E}[
 \sup\limits_{0\leq t \leq T}| Y_{t}^{2}|^{2}]<\infty.$$
Let $$f^{n}(t,y_{t}^{n},z_{t}^{n})\doteq
f(t,y_{t}^{n-1},z_{t}^{n-1})-C(y_{t}^{n}-y_{t}^{n-1})-\phi(|z_{t}^{n}-z_{t}^{n-1}|).$$
Then, from $(H5),(H6)$ and $(\ref{8})$ it follows that
 \begin{eqnarray*}
& &\mid f^{n}(t,y_{t}^{n},z_{t}^{n})| \leq |
f(t,y_{t}^{n-1},z_{t}^{n-1})\mid+C|
  y_{t}^{n}-y_{t}^{n-1}|+\phi(|z_{t}^{n}-z_{t}^{n-1}|)\\
  &\leq&\sum\limits_{i=1}^{2}\Big[| f_{i}(t,Y_{t}^{i},Z_{t}^{i})|+C|
  Y_{t}^{i}|+C|Z_{t}^{i}|\Big]\\
  &&+C\Big[| y_{t}^{n}|+|z_{t}^{n}|\Big]+3C\Big[|
  y_{t}^{n-1}|+|z_{t}^{n-1}|+1\Big].
\end{eqnarray*}
Thanks to $(H1)$ we get
\begin{eqnarray*}
|g(t,y_{t}^{n},z_{t}^{n})|^{2}&\leq&(1+\frac{1-\alpha}{2\alpha})|
g(t,y_{t}^{n},z_{t}^{n})-g(t,0,0)\mid^{2}
+(1+\frac{2\alpha}{1-\alpha})\mid g(t,0,0)\mid^{2}\\
 &\leq& \frac{1+\alpha}{2\alpha}C|y_{t}^{n}|^{2}+\frac{1+\alpha}{2}|z_{t}^{n}|^{2}
 +\frac{1+\alpha}{1-\alpha}\mid g(t,0,0)\mid^{2}.
\end{eqnarray*}
We apply It\^o's formula to $\mid y_{t}^{n}\mid^{2}$  and obtain
 \begin{eqnarray*}
\mathbb{E}\int_{0}^{T}|z_{t}^{n}|^{2}dt&=&\mathbb{E}\mid\xi\mid^{2}-\mid
y_{0}^{n}\mid^{2}
  +2\mathbb{E}\int_{0}^{T}y_{t}^{n}f^{n}(t,y_{t}^{n},z_{t}^{n})dt\\
  &&+\mathbb{E}\int_{0}^{T}\mid g(t,y_{t}^{n},z_{t}^{n})\mid^{2}dt\\
   &\leq& C_{1}+\frac{3+\alpha}{4}\mathbb{E}\int_{0}^{T}|z_{t}^{n}|^{2}dt
   +\frac{1-\alpha}{8}E\int_{0}^{T}|z_{t}^{n-1}|^{2}dt,
\end{eqnarray*}
where
 \begin{eqnarray*}
 C_{1}&\doteq &\sup\limits_{n}\mathbb{E}\Big\{2\int_{0}^{T}
 \sum\limits_{i=1}^{2} |y_{t}^{n}|\Big[|f_{i}(t,Y_{t}^{i},Z_{t}^{i})|
 +C|Y_{t}^{i}| +C|Z_{t}^{i}|\Big] dt\\
 &&+\frac{1+\alpha}{1-\alpha}\mathbb{E}\int_{0}^{T}|g(t,0,0)|^{2}dt
 +\Big(2C+\frac{1+\alpha}{2\alpha}C+\frac{76C^{2}}{1-\alpha}\Big)\int_{0}^{T}|y_{t}^{n}|^{2}dt\\
 &&+
 6C\int_{0}^{T}|y_{t}^{n}y_{t}^{n-1}|dt +6C\int_{0}^{T}|y_{t}^{n}|dt\Big\}+\mathbb{E}|\xi|^{2}<\infty.
\end{eqnarray*}
Then, we deduce
$$\mathbb{E}\int_{0}^{T}|z_{t}^{n}|^{2}dt\leq\frac{ 4C_{1}}{1-\alpha}+\frac{1}{2}E\int_{0}^{T}|z_{t}^{n-1}|^{2}dt.$$
Therefore, we get
 \begin{eqnarray*}
 \sup\limits_{n}\mathbb{E}\int_{0}^{T}|z_{t}^{n}|^{2}dt<\infty
\end{eqnarray*}
and
 \begin{eqnarray*}
   \sup\limits_{n}\mathbb{E}\int_{0}^{T}|f^{n}(t,y_{t}^{n},z_{t}^{n})|^{2}dt<\infty.
\end{eqnarray*}
Let $$C_{2}\doteq
\sup\limits_{n}\mathbb{E}\int_{0}^{T}|f^{n}(t,y_{t}^{n},z_{t}^{n})|^{2}dt.$$
Using It\^o's formula to $\mid y_{t}^{n}-y_{t}^{m}\mid^{2}$   we
obtain
\begin{eqnarray*}
 &&\mathbb{E}\int_{0}^{T}|z_{t}^{n}-z_{t}^{m}|^{2}dt+\mid
 y_{0}^{n}-y_{0}^{m}\mid^{2}\\
  &=&2\mathbb{E}\int_{0}^{T}(y_{t}^{n}-y_{t}^{m})\Big(f^{n}(t,y_{t}^{n},z_{t}^{n})-f^{m}(t,y_{t}^{m},z_{t}^{m})\Big)dt\\
  &&+\mathbb{E}\int_{0}^{T}\mid
  g(t,y_{t}^{n},z_{t}^{n})-g(t,y_{t}^{m},z_{t}^{m})\mid^{2}dt.
\end{eqnarray*}
Due to  $(H1)$  again it follows that
\begin{eqnarray*}
 &&\mathbb{E}\int_{0}^{T}|z_{t}^{n}-z_{t}^{m}|^{2}dt+\mid
 y_{0}^{n}-y_{0}^{m}\mid^{2}\\
   &\leq&
   4C_{2}^{\frac{1}{2}}\Big\{\mathbb{E}\int_{0}^{T}|y_{t}^{n}-y_{t}^{m}|^{2}dt\Big\}^{\frac{1}{2}}\\
   &&+\alpha \mathbb{E}\int_{0}^{T}|z_{t}^{n}-z_{t}^{m}|^{2}dt
   +C\mathbb{E}\int_{0}^{T}|y_{t}^{n}-y_{t}^{m}|^{2}dt.
\end{eqnarray*}
Then, we have
\begin{eqnarray*}
  &&(1-\alpha)\mathbb{E}\int_{0}^{T}|z_{t}^{n}-z_{t}^{m}|^{2}dt\\
   &&\leq 4C_{2}^{\frac{1}{2}}\Big\{\mathbb{E}\int_{0}^{T}|y_{t}^{n}-y_{t}^{m}|^{2}dt\Big\}^{\frac{1}{2}}
   +C\mathbb{E}\int_{0}^{T}|y_{t}^{n}-y_{t}^{m}|^{2}dt.
\end{eqnarray*}
Therefore, $\{z^{n}\}_{n=1}^{\infty}$ is a Cauchy sequence in
$M^{2}(0,T;\mathbb{R}^{d})$, and there exists $\underline{z}\in
M^{2}(0,T;\mathbb{R}^{d})$ such that
\begin{eqnarray*}
\lim\limits_{n\rightarrow\infty}
E\int_{0}^{T}|z_{t}^{n}-\underline{z}_{t}|^{2}dt=0.
\end{eqnarray*}
From $(H1),(H2),(H5), (H6)$, the above equality and
$\{y^{n}\}_{n=1}^{\infty}$ converges to $\underline{y}$ in
 $S^{2}(0,T;\mathbb{R})$ it follows that
\begin{eqnarray}\label{e1}
&&\sup\limits_{t\in
[0,T]}|\int_{t}^{T}z_{s}^{n}dW_{s}-\int_{t}^{T}\underline{z}_{s}dW_{s}|\stackrel{\mathbb{P}}{\longrightarrow}0,
\end{eqnarray}
\begin{eqnarray}\label{e2}
&&\sup\limits_{t\in
[0,T]}|\int_{t}^{T}g(s,y_{s}^{n},z_{s}^{n})dB_{s}
-\int_{t}^{T}g(s,\underline{y}_{s},\underline{z}_{s})dB_{s}|\stackrel{\mathbb{P}}{\longrightarrow}0,
\end{eqnarray}
and for almost all $\omega\in\Omega$, passing to a subsequence if
necessary, we have
\begin{eqnarray*}
f^{n}(t,y_{t}^{n},z_{t}^{n})
-f(t,\underline{y}_{t},\underline{z}_{t})\rightarrow0,\ dt-a.e., \
\text{as} \ n\rightarrow\infty.
\end{eqnarray*}
Combining the above inequalities with the dominated convergence
theorem yield
\begin{eqnarray}\label{e3}
\int_{0}^{T}f^{n}(s,y_{s}^{n},z_{s}^{n})ds
\rightarrow\int_{0}^{T}f(s,\underline{y}_{s},\underline{z}_{s})ds,
\end{eqnarray}
as $n\rightarrow\infty$. Consequently, (\ref{e1}), (\ref{e2}) and
(\ref{e3}) allow us to pass to the limit on both sides of BDSDE
(\ref{6}), passing to a subsequence if necessary, it follows that
 \begin{eqnarray*}
\underline{y}_{t}=\xi+\int_{t}^{T}f(s,\underline{y}_{s},\underline{z}_{s})ds
+\int_{t}^{T}g(s,\underline{y}_{s},\underline{z}_{s})dB_{s}-\int_{t}^{T}\underline{z}_{s}dW_{s},
\ t\in [0,T].
\end{eqnarray*}
Consequently, BDSDE  with data $(f, g, T, \xi)$ has a solution
$(\underline{y},\underline{z})$.

Let $(y,z)$ be any solution of BDSDE (\ref{1}). From $f_{1}(t, y, z)
\leq f(t, y, z),$ for all $ (t,y,z)\in [0,T]\times\mathbb{R}\times
\mathbb{R}^{d}$, and Theorem \ref{th2} it follows that
$Y_{t}^{1}\leq y_{t},$ a.s., for all $t\in[0,T].$

For $n=1$, we consider the following BDSDE:
 \begin{eqnarray*}
 \begin{aligned}
 y_{t}^{1}=\xi&+\int_{t}^{T}\Big[f(s,Y_{s}^{1},Z_{s}^{1})-C(y_{s}^{1}-Y_{s}^{1})
 -\phi(|z_{s}^{1}-Z_{s}^{1}|)\Big]ds  \\
&+\int_{t}^{T}g(s,y_{s}^{1},z_{s}^{1})dB_{s}-\int_{t}^{T}z_{s}^{1}dW_{s}
,\  t\in [0,T].
\end{aligned}
\end{eqnarray*}
From $(H5)$, $(H6)$ and $Y^{1}\leq y$ it follows that
 $$f(t,y_{t},z_{t})\geq f(t,Y_{t}^{1},Z_{t}^{1})
  -C(y_{t}-Y_{t}^{1})- \phi(| z_{t}-Z_{t}^{1}|).$$
Thus, by virtue of Theorem \ref{th2} we have  $y^{1}_{t}\leq y_{t}$,
a.s., for all $t\in[0,T]$.

 For $n\geq 2$, we assume that $ y^{n}\leq y $.
 Let us consider the following BDSDE:
 \begin{eqnarray*}
  \begin{aligned}
 y_{t}^{n+1}=\xi&+\int_{t}^{T}\Big[f(s,y_{s}^{n},z_{s}^{n})-C(y_{s}^{n+1}-y_{s}^{n})
 -\phi(|z_{s}^{n+1}-z_{s}^{n}|)\Big]ds \\
&+\int_{t}^{T}g(s,y_{s}^{n+1},z_{s}^{n+1})dB_{s}-\int_{t}^{T}z_{s}^{n+1}dW_{s}
,\  \ t\in [0,T].
 \end{aligned}
\end{eqnarray*}
By the similar argument as $n=1$ we get
$$f(t,y_{t},z_{t})\geq f(t,y_{t}^{n},z_{t}^{n})
  -C(y_{t}-y_{t}^{n})- \phi(| z_{t}-z_{t}^{n}|).$$
Therefore,  using Theorem \ref{th2} again  we obtain $ y^{n+1}_{t}
\leq y_{t}$, a.s., for  all $t\in[0,T]$. Then, by virtue of the
first part of the proof and taking the limit we have  $
\underline{y}_{t} \leq y_{t}$ a.s., for all $t\in[0,T]$. The proof
is complete.
\end{proof}

\begin{remark}
The above theorem generalizes the result in Lin \cite{LP2008}. In
fact, we can take $\phi(x)=Cx, x\geq 0,$ where $C$ is a positive
constant.
\end{remark}

 \begin{remark}\label{12}
 Under the assumptions of Theorem \ref{10}, if $f_{2}$ satisfies $(H4)$ and $(H10)$,
 and  BDSDE (\ref{6}) is replaced by the following BDSDE:
\begin{eqnarray*}
 \begin{aligned}
 y_{t}^{n}=\xi&+\int_{t}^{T}[f(s,y_{s}^{n-1},z_{s}^{n-1})-C(y_{s}^{n}-y_{s}^{n-1})+\phi(|z_{s}^{n}-z_{s}^{n-1}|)]ds \\
&+\int_{t}^{T}g(s,y_{s}^{n},z_{s}^{n})dB_{s}-\int_{t}^{T}z_{s}^{n}dW_{s}
,\ \ t\in [0,T],
 \end{aligned}
\end{eqnarray*}
where $n=1,2,\cdots,$ and $(y^{0},z^{0})=(Y^{2},Z^{2})$.
 Similar to the proof of Lemma  \ref{7} and Theorem
 \ref{10}, we can prove that BDSDE with data $(f, g,
T, \xi)$ has the maximal solution.
\end{remark}

\begin{remark}  Under assumptions $(H5),(H6)$ and $(H9)$, the
 solution of BDSDE  with data $(f, g, T, \xi)$ may be non-unique.
 Let us consider the following BDSDE:
\begin{eqnarray*}
y_{t}=\int_{t}^{T}\Big[4s Sgn(y_{s})\sqrt{|
y_{s}|}+\sqrt{z_{s}1_{z_{s}\geq
0}}\Big]ds+\int_{t}^{T}\Big[1_{\{y_{s}<0\}}y_{s}
+\frac{1}{2}z_{s}\Big]dB_{s} -\int_{t}^{T}z_{s}dW_{s}, \ t\in[0,T],
\end{eqnarray*}
where $Sgn(x)= 1,  x\geq 0; \ Sgn(x)= -1,x< 0.$
 We can check that the above equation satisfies $(H5),(H6)$ and
$(H9)$, where$$f_{1}(t, y,z)=-2t^{2}-2| y|+z  \ \text{and} \
f_{2}=2t^{2}+2| y|+z.$$ It's easy to check that,  for each $c\in
[0,T]$ and $t\in [0,T]$, $(y_{t},z_{t})=(0,0)$ and
$(y_{t},z_{t})=\Big([max\{c^{2}-t^{2},0\}]^{2},0 \Big)$ are
solutions of the above BDSDE.
\end{remark}

Finally, we give a comparison theorem for BDSDEs with discontinuous
coefficients.

\begin{theorem}\label{9}
 We suppose that $f^{1}$ and $f^{2}$ satisfy  $(H5),(H6)$ and $(H9)$, and
$f_{1}$ satisfies $(H4)$ and $(H10)$. Let  the minimal solutions
$(\underline{y}^{1},\underline{z}^{1})$ and
$(\underline{y}^{2},\underline{z}^{2})$ of BDSDEs (\ref{1}) with
data $(f^{1},g,T,\xi^{1})$ and $(f^{2},g,T,\xi^{2})$, respectively.
If $\xi^{1}\leq \xi^{2}$, $a.s.$, and $f^{1}(t, y, z) \leq f^{2}(t,
y, z)$, $a.s.$, then we have $\underline{y}_{t}^{1} \leq
\underline{y}_{t}^{2}$, $ a.s.$, for all $\ t \in [0,T]$.
\end{theorem}

\begin{proof}
From  $(H9)$ we know that there exists the following BDSDE:
\begin{eqnarray*}
Y_{t}^{1}=\xi^{1}+\int_{t}^{T}f_{1}(s,Y_{s}^{1},Z_{s}^{1})ds
+\int_{t}^{T}g(s,Y_{s}^{1},Z_{s}^{1})dB_{s}-\int_{t}^{T}Z_{s}^{1}dW_{s}
,\   t\in [0,T],
\end{eqnarray*}
such that $f_{1}(t,Y_{t}^{1},Z_{t}^{1})\in M^{2}(0,T;\mathbb{R})$.

  We consider  a sequence of BDSDEs as follows:
\begin{eqnarray*}
\begin{aligned}
 y_{t}^{n}=\xi&+\int_{t}^{T}\Big[f^{1}(s,y_{s}^{n-1},z_{s}^{n-1})-C(y_{s}^{n}-y_{s}^{n-1})
 -\phi(|z_{s}^{n}-z_{s}^{n-1}|)\Big]ds \\
&+\int_{t}^{T}g(s,y_{s}^{n},z_{s}^{n})dB_{s}-\int_{t}^{T}z_{s}^{n}dW_{s}
,\   \ t\in [0,T],
\end{aligned}
\end{eqnarray*}
where $n=1,2,\cdots$ and $(y^{0},z^{0})=(Y^{1},Z^{1}).$

From $(H9)$ and  $f^{1}(t, y, z) \leq f^{2}(t, y, z)$, $a.s.$, we
know that  $f_{1}(t, y, z) \leq f^{2}(t, y, z)$, $a.s.$ Then, from
Theorem \ref{th2} it follows that  $Y_{t}^{1}\leq
\underline{y}_{t}^{2}$, $a.s.$, for all $t\in [0,T]$.

 For $n=1$, by virtue of
$(H5)$, $(H6)$ and $Y^{1}\leq
 \underline{y}^{2}$ we get
 $$f^{2}(t,\underline{y}_{t}^{2},\underline{z}_{t}^{2})-f^{1}(t,Y_{t}^{1},Z_{t}^{1})\geq
 f^{1}(t,\underline{y}_{t}^{2},\underline{z}_{t}^{2})-f^{1}(t,Y_{t}^{1},Z_{t}^{1})
 \geq -C(\underline{y}_{t}^{2}-Y_{t}^{1})- \phi(|
 \underline{z}_{t}^{2}-Z_{t}^{1}|).$$
 Then, we have
 $$f^{2}(t,\underline{y}_{t}^{2},\underline{z}_{t}^{2})
 \geq f^{1}(t,Y_{t}^{1},Z_{t}^{1}) -C(\underline{y}_{t}^{2}-Y_{t}^{1})- \phi(|
 \underline{z}_{t}^{2}-Z_{t}^{1}|).$$
Thus, by virtue of Theorem \ref{th2} we have $y^{1}_{t}\leq
\underline{y}_{t}^{2}$, a.s., for all $t\in[0,T]$.

 For $n\geq2$, we suppose that $ y^{n}\leq \underline{y}^{2}$.
 Then, let us consider the following BDSDE:
 \begin{eqnarray*}
 \begin{aligned}
 y_{t}^{n+1}=\xi&+\int_{t}^{T}\Big[f^{1}(s,y_{s}^{n},z_{s}^{n})-C(y_{s}^{n+1}-y_{s}^{n})
 -\phi(|z_{s}^{n+1}-z_{s}^{n}|)\Big]ds  \\
&+\int_{t}^{T}g(s,y_{s}^{n+1},z_{s}^{n+1})dB_{s}-\int_{t}^{T}z_{s}^{n+1}dW_{s}
,\ t\in [0,T].
\end{aligned}
\end{eqnarray*}
By virtue of the similar argument as $n=1$  we have
  $$f^{2}(t,\underline{y}_{t}^{2},\underline{z}_{t}^{2})
 \geq f^{1}(t,y_{t}^{n},z_{t}^{n}) -C(\underline{y}_{t}^{2}-y_{t}^{n})- \phi(|
 \underline{z}_{t}^{2}-z_{t}^{n}|).$$
Then, thanks to Theorem \ref{th2} we get $y^{n+1}_{t}\leq
\underline{y}_{t}^{2}$, a.s., for all $t\in[0,T]$. From the proof of
Theorem \ref{10}  it follows that $\underline{y}_{t}^{1}\leq
\underline{y}_{t}^{2}$, a.s., for all $t\in[0,T]$. The proof is
complete.
\end{proof}

\begin{remark} Similar to the proof of Theorem
 \ref{9}, we can prove that a comparison theorem for the maximal solution of BDSDE
 with data $(f,g, T,\xi)$  by using Remark \ref{12}.
 \end{remark}

\end{document}